\documentclass[leqno]{amsart}
\usepackage{amssymb}
\usepackage{mathrsfs}
\usepackage{amsmath, amsfonts, vmargin, enumerate}
\usepackage{graphics}
\usepackage{verbatim}
\usepackage{amsthm}
\usepackage{latexsym,bm}
\usepackage{euscript}
\usepackage{dsfont}
\usepackage{color}

\input xypic
\xyoption {all}

 \makeatletter

\newtheorem{thm}{Theorem}[section]

\newtheorem{lem}[thm]{Lemma}
\newtheorem{defi}[thm]{Definition}
\newtheorem{remark}[thm]{Remark}
\newtheorem{example}[thm]{Example}
\newtheorem{pb}[thm]{Problem}

\numberwithin{equation}{section}

\newcommand{\nat}{{\mathbb N}}

\newcommand{\B}{{\mathcal B}}

\newcommand{\M}{{\mathcal M}}

\newcommand{\U}{{\mathcal U}}

\newcommand{\e}{\varepsilon}

\newcommand{\8}{\infty}

\newcommand{\ep}{\varepsilon}

\newcommand{\be}{\begin{eqnarray*}}
\newcommand{\ee}{\end{eqnarray*}}
\newcommand{\beq}{\begin{equation}}
\newcommand{\eeq}{\end{equation}}
\newcommand{\beqn}{\begin{equation*}}
\newcommand{\eeqn}{\end{equation*}}

\begin{document}

\title{Noncommutative Ergodic Theorems for Connected Amenable Groups}


\author{Mu Sun}

\maketitle

\begin{abstract}
This paper is devoted to the study of noncommutative ergodic theorems for connected amenable locally compact groups. For a dynamical system $(\M,\tau,G,\sigma)$, where $(\M,\tau)$ is a von Neumann algebra with a normal faithful finite trace and $(G,\sigma)$ is a connected amenable locally compact group with a well defined representation on $\M$, we try to find the largest noncommutative function spaces constructed from $\M$ on which the individual ergodic theorems hold. By using the Emerson-Greenleaf's structure theorem, we transfer the key question to proving the ergodic theorems for $\mathbb{R}^d$ group actions. Splitting the $\mathbb{R}^d$ actions problem in two cases according to different multi-parameter convergence types---cube convergence and unrestricted convergence, we can give maximal ergodic inequalities on $L_1(\M)$ and on noncommutative Orlicz space $L_1\log^{2(d-1)}L(\M)$, each of which is deduced from the result already known in discrete case. Finally we give the individual ergodic theorems for $G$ acting on $L_1(\M)$ and on $L_1\log^{2(d-1)}L(\M)$, where the ergodic averages are taken along certain sequences of measurable subsets of $G$.
\end{abstract}

\section{Introduction}\label{intro}
The study of ergodic theorems is an old branch of dynamical system theory which was started in 1931 by von Neumann and Birkhoff, having its origins in statistical mechanics. While new applications to mathematical physics continued to come in, the theory soon earned its own rights as an important chapter in functional analysis and probability. In the classical situation the stationarity is described by a measure preserving transformation $T$, and one considers averages taken along a sequence $f,~f\circ T,~ f\circ T^2,\, \dots$ for integrable $f.$ This corresponds to the probabilistic concept for stationarity. The modes of convergence under consideration mostly are norm convergence for``mean'' ergodic theorems, and convergegence almost everywhere for ``individual'' (or pointwise) ergodic theorems. More generally, we study semigroups (or groups) $\{T_g, g\in G\}$ of operators and limits of averages of $T_g f$ over subsets $I_n \subset G.$ In particular, group-theoretic considerations seem to be inherent in this analysis subject, and usually more interesting according to the abundance of structures for different group types. Given a locally compact group G with left Haar measure $|\cdot|$, a sequence of measurable subsets $\{F_n\}_{n=1}^\8$ in $G$ is called a F{\o}lner sequence (similarly for F{\o}lner net), if $0<|F_n|<\8$ for each $n$ and $$\forall g \,\in \, G,\quad \lim_{n\rightarrow \8} \frac{|gF_n \bigtriangleup F_n|}{|F_n|}=0.$$ A locally compact group is called amenable if it has a F{\o}lner sequence. There are many other equivalent ways to define amenability, however in the context of ergodic theory, the above definition is the most convenient. One of the ultimate aims in this direction is to establish the ergodic theorems for amenable group which is treated as a substitute for Abelian group. Using the group structure theorem, Emerson and Greenleaf gave a result for connected amenable locally compact group in \cite{EG1974}. It is until 2001 that Lindenstrauss's paper in {\it Inventiones} \cite{Lin2001} gave the pointwise ergodic theorem for the general amenable group and the averages is taken along certain tempered F{\o}lner sequence. The research on the existence and choice of the convergence sequence in the group is also quite an interesting and delicate subject.

While the observable quantities in classical systems can be described by real functions on probability space, they can be given for quantum systems by the operators in noncommutative probability space. In this case we have the so called noncommutative ergodic theory which has been developed since the very beginning of the theory of ``rings of operators''. But only mean ergodic theorems have been obtained at the early stage. It is until 1976 that a substitute of ``a.e. convergence'' was first treated in Lance's work \cite{la-erg}, which is called almost uniform convergence. This conception unsealed the study of individual ergodic theorems for noncommutative case. In the same period, Conze and Dang-Ngoc was inspired by Lance's work and generalized the pointwise ergodic result in \cite{EG1974} to the actions on von Neumann algebra, which is now the $p=\8$ case of noncommutative $L_p$ spaces.

In this paper, we denote $\mathcal{M}$ as a von Neumann algebra equipped with a normal faithful finite
trace $\tau$. In this case, we will often assume that $\tau$ is normalized, i.e., $\tau(1)=1,$ thus we call $(\M,\tau)$ a noncommutative probability space. Let $L_p(\M)(1\le p\le\8)$ be the associated noncommutative $L_p$-space and $L_p\log^r L(\M)(r>0)$ the associated noncommutative Orlicz space. The set of all projections in $\M$ is denoted as $P(\M).$ A multi-parameter net $\{x_{\alpha_1,\dots,\alpha_d}\}$ is said to converge bilaterally almost uniformly (resp. almost uniformly) to ${x}$, if for any $\e>0$, there exists $e\in P(\M)$, such that
$$ \tau(e^\bot)\le \e \text{ and } \{e(x_{\alpha_1,\dots,\alpha_d}-x)e\} ~(\text{resp. }\{(x_{\alpha_1,\dots,\alpha_d}-x)e\})$$
converges to $0$ in $\M$. Usually we denote it as b.a.u. (resp. a.u.) convergence.

Let $T:\mathcal{M} \rightarrow  \mathcal{M}$ be a linear map, it is called a kernel if it is a positive contraction and satisfy $\tau \circ T \leq \tau $ for all $x \in L_1(\mathcal{M}) \cap \mathcal{M}_+$. We know the domain of kernels can naturally extend to $L_p(\mathcal{M})$ for any $1\le p<\8$ (c.f. Lemma 1.1 \cite{JX2007}), and we denote $\U$ as the set of all kernels.

Now given a locally compact semigroup $G$, we define a homomorphism $\sigma:G \to \U$ where correspondingly $g \mapsto T_g$ as $G$-actions on $L_p(\mathcal{M})$, and the mapping $g \to T_g(x)$ is strongly continuous, for any $x \in L_p(\mathcal{M})$, $1\le p<\8$. So putting together $(\mathcal{M},\tau, G,\sigma)$ we call it a noncommutative dynamical system, and we
we denote by $F_\sigma$ the projection from $L_p(\mathcal{M})$ onto the invariant subspace of $G$-actions. Ergodic averages along increasing measurable net $\{V_\alpha\}$ is defined as $$M_{V_\alpha}(x)=\frac{1}{|V_\alpha|}\int_{V_\alpha}T_g(x)\,dg.$$

Our aim is to further generalize the noncommutative individual ergodic theorem for the system $(\M,\tau, G, \sigma)$ if $G$ is a connected amenable locally compact group. The main step of our method is also inherited from \cite{EG1974} that we decompose the connected amenable Lie group into the product of one compact group and $\mathbb{R}^d$ group. However, it is usually difficult to use iteration to obtain the result for $p \sim 1$ case for multi-parameter $\mathbb{R}$ actions.

To describe the situation, we here introduce two types of convergence for multi-parameter sequence $\{a_{k_1,\dots,k_d}\}_{(k_1,\dots,k_d)\in \mathbb{N}^d}$ (similar for net) in any Banach space $\B$. The first type is called cube convergence, if the sequence $ \{a_{n,\dots,n}\}_{n\in\mathbb{N}}$ converges in $\B$ when $n$ tends to $\8.$ The second type is called unrestricted convergence, if $\{a_{k_1,\dots,k_d}\}_{(k_1,\dots,k_d)\in \mathbb{N}^d}$ converges in $\B$ when $k_1,\dots,k_d$ tend to $\8$ arbitrarily. Thus when we consider the ergodic averages for $\mathbb{R}^d$ actions, we have separate maximal ergodic inequalities according to different types of convergence, for which we refer to section 3. We emphasize here that for the more general case of unrestricted convergence, it is kind of a breakthrough since there is no weak type $(1,1)$ inequality for multi-parameter noncommutative dynamical system. It is only recently in my joint work with Hong \cite{HS2016} the maximal inequality for the case of $\mathbb{Z}^d$ group acting on noncommutative Orlicz space $L_1\log^{2(d-1)}L(\M)$ is obtained, in the light of which we can give a result in continuous case.

It is in section 2 we give some more detailed introduction of noncommutative probability space, and in section 4 we present the main result of this paper. Since it is obvious that cube convergence is just a special case of unrestricted convergence, the sequences of measurable subsets of $G$ taken along for keeping cube convergence of  the ergodic averages should be included in unrestricted convergence case. This idea is specified through the narrative and  proof Theorem \ref{th:pointwisegodic}.

\section{Preliminaries}\label{preli}

We use standard notions for the theory of noncommutative $L_p$ spaces. Our main references are \cite{PX2003} and \cite{Xu07}.
Let $\mathcal{M}$ be a von Neumann algebra equipped with a normal faithful finite trace $\tau$. Let $L_0(\M)$ be the spaces of measurable operators associated to $(\M,\tau)$. For a measurable operator $x$, its generalized singular number is defined as
$$\mu_t(x) = \inf \{ \lambda > 0 : \tau \big( \mathds{1}_{(\lambda,\8)}(|x|)\big)\le t\}, \ \ ~ t>0.$$
The trace $\tau$ can be extended to the positive cone $L_0^+(\M)$ of $L_0(\M)$, still denoted by $\tau$,
$$\tau(x) = \int_0^\8 \mu_t(x) dt, \ \ ~ x\in L_0^+(\M).$$
Given $0 < p < \8,$ let
$$L_p(\M) = \{x \in L_0(\M) : \tau (|x|^p) < \8\}$$
and for $x \in L_p(\M),$
$$\|x\|_p = \big(\tau (|x|^p)\big)^{\frac1p} = \big(\int_0^\8 (\mu_t(x))^p dt \big)^{\frac1p}.$$
Then $(L_p(\M),\|\cdot\|_p)$ is a Banach space (or quasi-Banach space when $p<1$). This is the noncommutative $L_p$ space associated with $(\M,\tau)$, denoted by $L_p(\M,\tau)$ or simply by $L_p(\M)$.  As usual, we set $L_\infty(\M,\tau)=\M$ equipped with the operator norm.

The noncommutative Orlicz spaces are defined in a similar way as commutative ones. Given an Orlicz function $\Phi$, the Orlicz
space $L_\Phi(\M)$ is defined as the set of all measurable operators $x$ such
that $\Phi( \frac{|x|}{\lambda} ) \in L_1(\M)$ for some $ \lambda>0$. Equipped with the norm
$$ \|x\|_\Phi = \inf\left\{ \lambda>0: \tau \left[ \Phi\left( \frac{|x|}{\lambda} \right) \right] \le 1\right\},$$
$L_\Phi(\M)$ is a Banach space.
When $\Phi(t)=t^p$ with $1\leq p<\infty$, the space $L_\Phi(\M)$ coincides with $L_p(\M)$. If $\Phi(t)= t^p(1+ \log^+t)^r$ for $ 1\le p <\8 $ and $r>0$, we get the space $L_p\log^r L(\M)$. From the definition, it is easy to check that whenever $(\M,\tau)$ is a probability space, we have
$$L_q(\M)\subset L_p\log^r L(\M)\subset L_s(\M)$$
for $q>p\geq s\geq1$.

\bigskip

The spaces $L_p(\M; \ell_\8)$ and $L_p(\M; \ell_\8^c)$ play an important role in the formulation of noncommutative maximal inequalities. A sequence
$ \{x_n\}_{n\ge 0} \subset L_p(\M)$ belongs to $L_p(\M; \ell_\8)$ if and only if it can be factored
as $x_n = a y_n b$ with $a, b \in L_{2p}(\M)$ and a bounded sequence $\{y_n\} \subset L_\8(\M)$. We
then define
$$\|\{x_n\}_n\|_{L_p(\ell_\8)}=\inf_{x_n = a y_n b}\big\{\|a\|_{2p}\,
 \sup_{n\ge0}\|y_n\|_\8\,\|b\|_{2p}\big\} .$$
Following \cite{JX2007}, this norm is symbolically denoted by $\|{\sup_n}^+ x_n\|_p$. It is shown in \cite{JX2007} that a positive
sequence $\{x_n\}_n$ belongs to $L_p(\M; \ell_\8)$ if and only if there exists
$a \in L^+_p (\M)$ such that $x_n \le a$ for all $n \ge 0$. Moreover

$$\|{\sup_n}^+ x_n\|_p = \inf\big\{\|a\|_{p}\;:\; a\in L^+_p(\M)
 \;\mbox{s.t.}\; x_n\le a,\ \forall\;n\ge0\big\}.$$
Here and in the rest of the paper, $L^+_p (\M)$ denotes the positive cone of $L_p(\M).$ The space
$L_p(\M; \ell_\8^c)$ is defined to be the set of sequences $\{x_n\}_{n\ge 0}$ with $\{|x_n|^2\}_{n\ge 0}$ belonging to $L_{p/2}(\M; \ell_\8)$ equipped with (quasi) norm
$$\|\{x_n\}_n\|_{L_p(\ell^c_\8)}=\|\{|x_n|^2\}_n\|^{\frac12}_{L_{\frac p2}(\ell_\8)}.$$
We refer to \cite{Jun2002}, \cite{JX2007} and \cite{Musat2003} for more information on these spaces and related facts.

The vector-valued Orlicz spaces $L_p\log^r L(\M; \ell_\8)$  $(1\le p< \8,\ r>0)$ are firstly introduced by Bekjan {\it et al} in \cite{BCO2013}. The key observation is that ${L_p(\ell_\8)}$-norm has an equivalent formulation:
$$\|\{x_n\}_n\|_{L_p(\ell_\8)}=\inf \big\{\frac12(\|a\|^2_{2p}+\|b\|^2_{2p})\,
 \sup_{n\ge0}\|y_n\|_\8\big\},$$
where the infimum is taken over the same parameter as before. Given an Orlicz function $\Phi$, let $\{x_n\}$ be a sequence of operators in $L_{\Phi}(\M)$, we define
$$\tau\big(\Phi({\sup_n}^+ x_n)\big)=\inf \big\{\frac12\big(\tau(\Phi(|a|^2))+\tau(\Phi(|b|^2))\big)\,
 \sup_{n\ge0}\|y_n\|_\8\big\},$$
where the infimum is taken over all the decompositions $x_n=ay_nb$ for $a,b\in L_0(\M)$ and $y_n\in L_\infty(\M)$ with $|a|^2,|b|^2\in L_\Phi(\M)$ and $\sup_n\|y_n\|_\infty\leq1$. Then $L_\Phi(\M;\ell_\infty)$ is defined to be the set of sequences $\{x_n\}_n\subset L_\Phi(\M)$ such that there exists one $\lambda>0$ satisfying $$\tau\big(\Phi({\sup_n}^+ \frac{x_n}\lambda)\big)<\infty$$
equipped with the norm
$$ \|\{x_n\}_n\|_{L_\Phi(\ell_\infty)} = \inf\left\{ \lambda>0: \tau\big(\Phi({\sup_n}^+ \frac{x_n}\lambda)\big)<1\right\}.$$
Then $(L_\Phi(\M;\ell_\infty),\|\cdot\|_{L_\Phi(\ell_\infty)})$ is a Banach space.
It was proved in \cite{BCO2013} that a similar characterization holds for sequences of positive operators:
$$\tau\big(\Phi({\sup_n}^+ x_n)\big) \thickapprox \inf\big\{\tau(\Phi(a))\;:\; a\in L^+_\Phi(\M)
 \;\mbox{s.t.}\; x_n\le a,\ \forall\;n\ge0\big\}$$
which implies a similar characterization for the norm
$$\|\{x_n\}_n\|_{L_\Phi(\ell_\infty)} \thickapprox \inf\big\{\|a\|_{\Phi}\;:\; a\in L^+_\Phi(\M)
 \;\mbox{s.t.}\; x_n\le a,\ \forall\;n\ge0\big\}.$$
From the definition, it is not difficult to check that whenever $(\M,\tau)$ is a probability space, we have
$$L_q(\M;\ell_\infty)\subset L_p\log^r L(\M;\ell_\infty)\subset L_s(\M;\ell_\infty)$$
for $q>p\geq s\geq1$. We refer the reader to \cite{BCO2013} for more information on the vector-valued Orlicz spaces.

\section{Maximal ergodic inequalities for multi-parameter $\mathbb{R}^+$ semigroups}

In this section, we are going to give two versions of maximal ergodic inequalities for $d$-parameter ($d\ge 2$) $\mathbb{R}^+$ semigroups according to different types of convergence.

We give the cube convergence version first which is directly a generalization of weak type $(1,1)$ inequality to multi-parameter and continuous case.

\begin{thm}\label{th:cubemaximalineq}
Let $(\M, \tau, \mathbb{R}^+, \sigma^i) \quad i=1,\dots,d$ be $d$ noncommutative dynamical systems. Set $$\sigma^i(s)= T^{(i)}_s  \quad i=1,\,\dots\,,d~\text{ and }~ M_{[t]}=\frac{1}{t^d}\int_0^{t}T^{(d)}_{s_d}\,ds_d
 \,\cdots \,\int_0^{t}T^{(1)}_{s_1}\,ds_1.$$ Then for every $x\in L_1(\M)$ and any $\lambda>0$, there exists $e\in P(\M)$ such that
 $$\tau(e^\bot)\le C \frac{\|x\|_1}{\lambda} \quad\text{and}\quad
 \sup_{t> 0}\|e\, M_{[t]}(x)\, e\|_\8 \le \lambda\,,$$ where $C$ is a positive constant independent of $x$ and $\lambda$.
\end{thm}

\begin{proof}
Recall that the semigroup $(T^{(i)}_s)_{s\ge 0}\quad i=1,\dots,d$ is strongly continuous on $L_1(\M)$, i.e., for any $x\in L_1(\M)$ the function $s\mapsto T^{(i)}_s(x)$ is
continuous from $[0,\8)$ to $L_1(\M)$, and so is the function
$t\mapsto M_{[t]}(x)$. Thus to prove the inequality it
suffices to consider $M_{[t]}(x)$ for $t$ in a dense subset of
$(0,\8)$, for instance, the subset $\{n2^{-m}\;:\; m, n\in
\nat\}$. Using once more the strong continuity of
$(T^{(i)}_s)_{s\ge0}$, we can replace the integral defining $M_{[t]}(x)$ by
a Riemann sum. Thus we have approximately
 \be
 M_{[n2^{-m}]}(x)
 &=&\frac{1}{(n2^{-m})^d}\sum_{k_d=0}^{n-1}\int_{k_d2^{-m}}^{(k_d+1)2^{-m}}T^{(d)}_{s_d}(x)\,d s_d \,\cdots\,\sum_{k_1=0}^{n-1}\int_{k_12^{-m}}^{(k_1+1)2^{-m}}T^{(1)}_{s_1}(x)\,d s_1\\
 &\approx& \frac{1}{n^d}\sum_{k_d=0}^{n-1}T^{(d)}_{k_d2^{-m}}(x)\,\cdots\, \sum_{k_1=0}^{n-1}T^{(1)}_{k_12^{-m}}(x).
 \ee
As we have transformed our semigroup ergodic averages into the discrete case which is already proved in Theorem 1.2 \cite{Ska2005}, thus we conclude the wanted weak type $(1,1)$ inequality.
\end{proof}

Next we give the unrestricted version. It is on the other hand an extension to continuous case of the multi-parameter maximal ergodic inequalities on noncommutative Orlicz spaces recently proved in \cite{HS2016}. Though we make loose of the trace preserving condition, the proof in \cite{HS2016} still hold in our kernel case. The rest is to use the same transference as in Theorem \ref{th:cubemaximalineq}, so we omit the proof here.

\begin{thm}\label{th:unrestrictedmaximalineq}
Let $(\M, \tau, \mathbb{R}^+, \sigma^i) \quad i=1,\,\dots\,,d$ be $d$ noncommutative dynamical systems. Set $$\sigma^i(s)= T^{(i)}_s \quad i=1,\,\dots\,,d~\text{ and }~ M_{t_1,...\,, t_d}=\frac{1}{t_1\,\cdots\, t_d}\,
 \int_0^{t_d}T^{(d)}_{s_d}\,ds_d
 \,\cdots \,\int_0^{t_1}T^{(1)}_{s_1}\,ds_1.$$ Then for every $x\in L_1\log^{2(d-1)} L(\mathcal{M})$ and any $\lambda>0$, there exists $e\in P(\mathcal{M})$ such that
$$\tau(e^{\bot}) \le
 C \frac{\| x \|_{L_1\log^{2(d-1)}L}}{\lambda} ~ \text{ and } ~ \sup_{t_1>0,\dots\,,t_d>0}\|e\, M_{t_1,\dots\,,t_d}(x)\, e\|_{\8} \le  \lambda\,,$$ where $C$ is a positive constant independent of $x$ and $\lambda$;
moreover, for $x\in L_2\log^{2(d-1)} L(\mathcal{M})$ we have the following estimates $$\tau(e^{\bot}) \le
 \left(C \frac{\| x \|_{L_2\log^{2(d-1)}L}}{\lambda}\right)^2 ~ \text{ and } ~ \sup_{t_1>0,\dots\,,t_d>0}\| M_{t_1,\dots\,,t_d}(x)\, e\|_{\8} \le \lambda.$$
\end{thm}

\section{Individual ergodic theorems}

In this section, we are going to treat different types of almost uniform convergence of ergodic averages for connected amenable group actions. As a preparation, we first give the following lemma which is a direct conclusion from last section.

\begin{lem}\label{le:semigroupindividualergodic}
Let $(\M, \tau, \mathbb{R}^+, \sigma^i) \quad i=1,\,\dots\,,d$ be $d$ noncommutative dynamical systems.
\begin{enumerate}[{\rm i)}]
\item For every $x\in L_1(\M)$, $M_{[t]}(x)$ converges b.a.u. to $F_{\sigma^d}\,\cdots\, F_{\sigma^1}(x)$ as $t$ tends to $\8$.
\item For every $x\in L_1\log^{2(d-1)} L(\mathcal{M})$, $M_{t_1,\dots\,,t_d}(x)$ converges b.a.u. to $F_{\sigma^d}\,\cdots\, F_{\sigma^1}(x)$ as $t_1, \dots\,, t_d$ tend to $\8$ arbitrarily. Moreover, if $x\in L_2\log^{2(d-1)} L(\mathcal{M})$, then we have the a.u. convergence.
\end{enumerate}
\end{lem}

\begin{proof}
We first prove ii). It is already known that for any $y\in L_\8(\M)$, $M_{t_1,\dots\,,t_d}(x)$ converges b.a.u. to $F_{\sigma^d}\,\cdots\, F_{\sigma^1}(x)$ (c.f. \cite{CN1978}, \cite{JX2007}), and $L_\8(\M)$ is dense in $L_1\log^{2(d-1)} L(\mathcal{M})$. Thus taking any $x\in L_1\log^{2(d-1)} L(\mathcal{M})$, any $\ep>0$ and any $\lambda>0,$ we can always find a $y \in L_\8(\M)$, where $\|x-y\|_{L_1\log^{2(d-1)}L(\M)} \le \frac{1}{2C}\lambda \ep$, $\|F_{\sigma^d}\,\cdots\, F_{\sigma^1}(x-y)\|_\8 \le \lambda$, and $C$ is the positive constant from the application of the first maximal inequality in Theorem \ref{th:unrestrictedmaximalineq} to the element $x-y$: there exists a projection $e_1\in P(\mathcal{M})$ such that $$\tau(e_1^{\bot}) \le
 C \frac{\| x-y \|_{L_1\log^{2(d-1)}L(\M)}}{\lambda}  \le  \frac{\ep}{ 2} ~ \text{ and } ~ \sup_{t_1>0,\dots\,,t_d>0}\|e_1\, M_{t_1,\dots\,,t_d}(x-y)\, e_1\|_{\8} \le  \lambda\,.$$

On the other hand, there exists a projection $e_2$, such that $\tau(e_2^{\bot})\le \ep \slash 2$ and $ \|e_2 \, (M_{t_1,\dots\,,t_d}(y)-F_{\sigma^d}\,\cdots\, F_{\sigma^1}(y))\, e_2 \|_\8 $ converges to $0$, which means that there exists $ r_1,\,\dots \,,r_d \in \mathbb{R}^+$, when $t_i\ge r_i,~i=1,\,\dots\,,d $, we have $$ \| e_2 M_{t_1,\dots\,,t_d}(y) e_2 - e_2\, F_{\sigma^d}\,\cdots\, F_{\sigma^1}(y)\, e_2 \|_{\8} < \lambda.$$

Now take $e = e_1\wedge e_2$, then we have $\tau(e^{\bot})\le \ep$ and
\be
\| e\, M_{t_1,\dots\,,t_d}(x) \, e - e\, F_{\sigma^d}\,\cdots\, F_{\sigma^1}(x)\, e \|_{\8}
\ee
\be\begin{split}
\le & \| e\, M_{t_1,\dots\,,t_d}(x-y)\, e \|_{\8} +\| e\, M_{t_1,\dots\,,t_d}(y)\, e - e\, F_{\sigma^d}\,\cdots\, F_{\sigma^1}(y)\, e \|_{\8} + \|e\, F_{\sigma^d}\,\cdots\, F_{\sigma^1}(y-x)\, e\|_{\8}\\
< & 3\lambda.
\end{split}\ee
So we have proved $M_{t_1,\dots\,,t_d}(x)$ converges b.a.u. to $F_{\sigma^d}\,\cdots\, F_{\sigma^1}(x)$ for every $x\in L_1\log^{2(d-1)} L(\mathcal{M})$, and using the same process with the second inequality in Theorem \ref{th:unrestrictedmaximalineq}, we have all the conclusions of ii).

As $L_1\log^{2(d-1)} L(\mathcal{M})$ is dense in $L_1(\M)$, and the cube convergence is just a special case of unrestricted convergence, then apply the maximal inequality in Theorem \ref{th:cubemaximalineq} and repeat the reasoning as above, it follows the result of i). Thus we complete the proof.
\end{proof}

We are now in position to give our main result by using a structure theorem of connected amenable locally compact group of Emerson and Greenleaf \cite{EG1974}.

\begin{thm}\label{th:pointwisegodic}
Given $(\M,\tau, G,\sigma)$ a noncommutative dynamical system with $G$ a connected amenable locally compact group, there exists $\mathcal{V}$ and $\mathcal{W}$ that are classes of increasing sequences of measurable subsets of $G$, $\mathcal{W}$ is strictly included in $\mathcal{V}$ such that:
\begin{enumerate}[{\rm i)}]
\item for each $(V_n)_{n\ge 1}\in \mathcal{V}$, $\displaystyle \mathop{\bigcup}_{n \ge 1} V_n $ generates $G$ and any $x\in L_1\log^{2(d-1)} L(\mathcal{M})$ (resp. $x\in L_2\log^{2(d-1)} L(\mathcal{M})$), we have
$$ \frac{1}{|V_n|} \int_{V_n} T_g(x) dg ~\text{converges to}~F_\sigma(x)~\text{b.a.u. (resp. a.u.)}~\text{when}~n\to \8;$$
\item for each $(W_n)_{n\ge 1}\in \mathcal{W}$ and any $x\in L_1(\M)$, we have
$$ \frac{1}{|W_n|} \int_{W_n} T_g(x) dg ~\text{converges to}~F_\sigma(x)~\text{b.a.u.}~\text{when}~n\to \8.$$

\end{enumerate}

\end{thm}

\begin{proof}
In any connected amenable locally compact group $G$, there is a largest compact normal subgroup K(G). The quotient $G'=G/K(G)$ is an amenable Lie group with $K(G')$ trivial.

\

a) We first treat the case of amenable Lie group $G'$ with $K(G')$ trivial.

Following Lemma 4.2 from \cite{EG1974} by Emerson and Greenleaf: there are closed one-parameter subgroups $L_1,\dots,L_d(L_i \simeq \mathbb{R})$ and a compact connected subgroup $K$ such that,

\begin{enumerate}[{\rm 1)}]

\item The map $j:K \times L_d \times \cdots \times L_1 \to G'$ is a homeomorphism;

\item The subsets $H_i= j(L_i \times \cdots \times L_1)$ are closed subgroups in $G'$, $1 \leq i \leq d$;

\item $H_i$ is normal in $H_{i+1}$ for $i=1,2,\dots,d$ (we take $H_{d+1}=G'$);

\item The map $j$ identifies the product Haar measures $m_K \times m_d \times \cdots \times m_1$ on $K \times L_d \times \cdots \times L_1$ with a left Haar measure $dg'$ on $G'$.

\end{enumerate}

Let $\sigma^i$ be the restriction of $\sigma$ to $L_i ~(i=1,\dots, d)$, and $\tilde{\sigma}$ the restriction to $K$. Therefore each subgroup forms a dynamical system, and for simplicity we denote by $F_i$ and $\tilde{F}$ the corresponding projections from $\mathcal{B}(\mathcal{M})$ onto the fixed point subspaces with respect to every subgroup actions. Here $\mathcal{B}(\mathcal{M})$ denotes some noncommutative function space constructed from $\M$. Without saying in particular, we usually identify group $G'$ with its parameterized form $K \times L_d \times \cdots \times L_1$.

Let $V'(t_d,\dots,t_1)=K \times [-t_d,t_d]\times \dots \times [-t_1,t_1], ~t_d,\dots,t_1>0$, in the following, our subsets sequence have this kind of form.

Choosing $\mathcal{B}(\mathcal{M})$ and the sequence $t_i(n_i) \xrightarrow[n_i \to \8]{} \8$, $i= 1,2, \dots, d$ properly, for any $x\in \mathcal{B}(\mathcal{M})$ we can achieve the following convergence

\beq\label{eq:convergence}
\frac{1}{|V'(t_d(n_d), \dots, t_1(n_1))|} \displaystyle \mathop{\int}_{V'(t_d(n_d), \dots, t_1(n_1))} T_{g'}(x) d g' \longrightarrow \tilde{F} F_d \cdots F_1(x)
\eeq
in different senses, and correspondingly we get the sequence classes $\mathcal{V}$ and $\mathcal{W}$ from the construction.

By the former decomposition properties we can easily deduce that $T_{g'}=\sigma(g')=\sigma(k,s_d,\dots, s_1)=\tilde{\sigma}(k)\, T^{(d)}_{s_d} \,\cdots\, T^{(1)}_{s_1} ,$ and we note

\be
M_{n_d,\dots\,, n_1}(x)=\frac{1}{2^d \,t_1 \cdots \, t_d} \displaystyle \mathop{\int}_{-t_d(n_d)}^{t_d(n_d)} \,d s_d\, \cdots \, \displaystyle \mathop{\int}_{-t_1(n_1)}^{t_1(n_1)} d s_1  T^{(d)}_{s_d}\, \cdots \, T^{(1)}_{s_1} (x)
\ee

\be
\tilde{M}_{n_d,\dots\,, n_1}(x)= \displaystyle \mathop{\int}_K \,dk \, \tilde{\sigma}(k) \,M_{n_d,\dots\,, n_1}(x)
\ee

Typically we consider

\be
&   & \Big \| \Big( \displaystyle \mathop{\int}_K \,\tilde{\sigma}(k)\, M_{n_d,\dots\,, n_1}(x)\, d k - \tilde{F} F_d \cdots F_1(x) \Big) e \Big \|_\8
\\
& = & \Big \|  \Big( \displaystyle \mathop{\int}_K\, dk\, \tilde{\sigma}(k) \, \big[ M_{n_d,\dots\,, n_1}(x) - F_d \cdots F_1(x) \big] \Big) e \Big \|_\8
\\
& \leq & \displaystyle \mathop{\int}_K \,dk \Big \| \Big( M_{n_d,\dots\,, n_1}(x) - F_d \cdots F_1(x) \Big) e \Big \|_\8 ,
\ee and $e$ denotes any projection in $P(\mathcal{M})$. So we know that the convergence of $\tilde{M}_{n_d,\dots\,, n_1}(x)$ is dominated by
$M_{n_d,\dots\,, n_1}(x)$.

Here is when we apply Lemma \ref{le:semigroupindividualergodic} to obtain the b.a.u. (resp. a.u.) convergence of (\ref{eq:convergence}) for every $x\in L_1\log^{2(d-1)} L(\mathcal{M})$ (resp. $x\in L_2\log^{2(d-1)} L(\mathcal{M})$) as $t_i(n_i) \xrightarrow[n_i \to \8]{} \8$, $i= 1,2,\, \dots \,, d$ arbitrarily;
in particular, if we let the sequence $t_1(n_1)=\,\cdots\,=t_d(n_d) \rightarrow \8$, then we have the b.a.u. convergence of (\ref{eq:convergence}) for every $x \in L_1(\M)$.

It remains to prove $ \tilde{F} F_d \cdots F_1(x) = F_\sigma(x) $. However, take any $h_1 \in H_1$ and $h\in H_2$. From the normality of $H_1$ in $H_2$, we know there exists $h'_1\in H_1$, such that $h_1 h = h h'_1$. Then we have $T_{h_1} T_{h}F_1(x)= T_{h}T_{h'_1}F_1(x)=T_{h}F_1(x)$. As $F_2F_1(x)$ is a limit(in strong operator topology) of combinations of $T_hF_1(x)$, so $F_2F_1(x)$ is invariant of $H_1$-actions. Similarly, we can prove that $ \tilde{F} F_d \cdots F_1(x)$ is invariant under each $L_1,\, \dots\, ,L_d,K$ subgroup actions, thus invariant of $G'$, so we finish this part of proof.

b) Now we return to the general case of $G$. Let $u$ be the canonical surjection of $G$ onto $G'=G/K(G)$, and define
$$V(t_d,\,\dots\,,t_1)=u^{-1}\big( V'(t_d,\,\dots\,,t_1)\big).$$

We also define $$\mathcal{V}=\{(V_n=V(t_d(n),\,\dots\,,t_1(n)))_{n\ge 1}: t_i(n)\rightarrow +\8 \text{ as }n \rightarrow +\8 ,\quad i=1,\,\dots\,,d\}$$
and $$ \mathcal{W}= \{(W_n)_{n\ge 1}\in \mathcal{V}: W_n=W(t_d(n),\,\dots\,,t_1(n)) \text{ satisfies } t_d(n)= \, \cdots \,= t_1(n) \text{ for each } n\}.$$

In the following, we can also write $K$ in short for $K(G)$ without confusion to part a) (since K is also compact subgroup) and let $dk$ be the normalized Haar measure on $K$, $dg$ be the left Haar measure on $G$ such that
\be
\displaystyle \mathop{\int}_G f(g)\, d g = \displaystyle \mathop{\int}_{G'} \,d g' \displaystyle \mathop{\int}_{K}\,f(g' k)\,dk, ~\quad~ \forall f \in C_0(G).
\ee

We have then
\be
|V(t_d,\,\dots\,,t_1)|=| V'(t_d,\,\dots\,,t_1)|.
\ee

Like before, for $K$ and $G'$, there are corresponding subgroup action systems as $(\mathcal{M},\tau, K, \tilde{\sigma})$ and $(\M,\tau, G', \sigma')$, also we denote $\tilde{F}$ and $F'$ as the corresponding projections onto each fixed point subspace.
Since $K$ is compact, we similarly consider the subgroup action, note that $\tilde{F}(x) = \displaystyle \mathop{\int}_{K}\, \tilde{\sigma}_{k}(x)\,dk$ and we have

\be
\frac{1}{|V(t_d, \,\dots\,, t_1)|} \displaystyle \mathop{\int}_{V(t_d, \,\dots\,, t_1)} T_g(x)\, d g = \frac{1}{|V'(t_d,\, \dots\,, t_1)|} \displaystyle \mathop{\int}_{V'(t_d,\, \dots\,, t_1)} \,d g' \sigma'_{g'}\, \tilde{F}(x).
\ee

Then by the result in part a), we get the corresponding convergence results of the above equation with respect to sequences in $\mathcal{V}$ and $\mathcal{W}$, and this completes the proof.

\end{proof}

\end{document}